\documentclass[a4paper,12pt]{amsart}
\usepackage{amsmath}
\usepackage{amssymb}
\usepackage{mathrsfs}
\usepackage{enumerate}
\usepackage{ifthen}
\usepackage{graphicx}
\usepackage{caption}
\usepackage[T1]{fontenc} %skandit
\usepackage{tikz}
\usepackage{cite}

\nonstopmode \numberwithin{equation}{section}
\newtheorem{theorem}{Theorem}[section]

\newtheorem{corollary}{Corollary}[section]

\newtheorem{problem}{Problem}[section]

\theoremstyle{remark}
\theoremstyle{definition}
\newtheorem{remark}{Remark}[section]
\newtheorem{definition}{Definition}[section]
\newtheorem{example}{Example}[section]

\theoremstyle{plain}
\newtheorem*{thmA}{Theorem A}
\newtheorem*{thmB}{Theorem B}
\newtheorem*{thmC}{Theorem C}

\newtheorem*{lemA}{Lemma A}

\numberwithin{equation}{section}
\numberwithin{theorem}{section}

\newcounter{minutes}
\divide\time by 60
\newcounter{hours}
\multiply\time by 60
\addtocounter{minutes}{-\time}
\begin{document}

\title{Growth, Distortion, Pre-Schwarzian and Schwarzian norm estimates for Generalized Robertson class}

\author{Molla Basir Ahamed$^*$}
\address{Molla Basir Ahamed, Department of Mathematics, Jadavpur University, Kolkata-700032, West Bengal, India.}
\email{mbahamed.math@jadavpuruniversity.in}

\author{ Rajesh Hossain}
\address{Rajesh Hossain, Department of Mathematics, Jadavpur University, Kolkata-700032, West Bengal, India.}
\email{rajesh1998hossain@gmail.com}


\def\thefootnote{}
\footnotetext{ {\tiny File:~\jobname.tex,
printed: \number\year-\number\month-\number\day,
          \thehours.\ifnum\theminutes<10{0}\fi\theminutes }
} \makeatletter\def\thefootnote{\@arabic\c@footnote}\makeatother

\begin{abstract}
Let $\mathcal{A}$ denote the class of analytic functions $f$ in the unit disc $\mathbb{D}=\{z\in\mathbb{C}:\;|z|<1\}$ normalized by $f(0)=0$ and $f^{\prime}(0)=1$. For $-\pi/2<\alpha<\pi/2$ and $0\leq \beta<1$ , let $\mathcal{SP_\alpha}(\beta)$ be the subclass of $\mathcal{A}$ defined by
\begin{align*}
	\mathcal{SP_\alpha}(\beta)=\bigg\{{\rm Re}\;\bigg\{e^{\iota\alpha}(1+\frac{zf^{\prime\prime}(z)}{f^{\prime}(z)}\bigg\}>\beta\cos\alpha\;\;\mbox{for}\;z\in\mathbb{D}\bigg\}.
\end{align*}
This paper investigates the geometric properties of functions belonging to the generalized Robertson class $\mathcal{SP_\alpha}(\beta)$, which consists of $\alpha$-starlike functions of order $\beta$. The primary objective is to derive sharp bounds for the norms of the Schwarzian and pre-Schwarzian derivatives for functions $f$ in this class. These bounds are expressed in terms of the initial coefficient $f^{\prime\prime}(0)$, with particular emphasis on the case where $f^{\prime\prime}(0)=0$. Additionally, we establish sharp distortion and growth theorems for the functions in $\mathcal{SP_\alpha}(\beta)$. Finally, we address the radius problem for this function class. Specifically, we determine the sharp radius of concavity and the sharp radius of convexity for functions in the class $\mathcal{SP_\alpha}(\beta)$. 
\end{abstract}
\maketitle
\pagestyle{myheadings}
\markboth{ Molla Basir Ahamed, Rajesh Hossain }{Pre-Schwarzian, Schwarzian norm estimates, Growth and Distortion theorems}

\section{\bf Introduction}\vspace*{1.5mm}
Let \( \mathbb{D} = \{z \in \mathbb{C} : |z| < 1\} \) be the unit disk, and define \( \mathcal{H} \) as the class of analytic functions on \( \mathbb{D} \). Let $\mathcal{S}^*(\alpha)$ and $\mathcal{C}(\alpha)$ denote respectively, the classes of starlike and convex functions of order $\alpha$ for $0\leq \alpha<1$ in $\mathcal{S}$. It is well-known that a function $f\in\mathcal{A}$ belongs to $\mathcal{S}^*(\alpha)$ if, and only if, ${\rm Re}(zf^{\prime}(z)/f(z))>\alpha$ for $z\in\mathbb{D}$, and $f\in\mathcal{C}(\alpha)$ if, and only if, ${\rm Re}(1+zf^{\prime\prime}(z)/f^{\prime}(z))>\alpha$. Similarly, a function $f\in\mathcal{A}$ belongs to $\mathcal{K}$, the class of close-to-convex functions, if and only if, there exists $g\in\mathcal{S}^*$ such that ${\rm Re}[e^{i\tau}(zf^{\prime}(z))/g(z)]>0$ for $z\in\mathbb{D}$ and $\tau\in (-\pi/2, \pi/2)$. Thus, it is easy to see that  $\mathcal{C}\subset\mathcal{S}^*\subset\mathcal{K}\subset\mathcal{S}$. In particular, when $\tau=0$, then the resulting subclass of the close-to-convex functions is denoted by $\mathcal{K}_0$. The subclass \( \mathcal{LU} \) consists of locally univalent functions, \textit{i.e.,} functions \( f \in \mathcal{H} \) with \( f^{\prime}(z) \neq 0 \) for all \( z \in \mathbb{D} \). For functions $f\in\mathcal{LU}$ defined in a simply connected domain ${\Omega}$, the pre-Schwarzian derivative $P_f$ and the Schwarzian derivative $S_f$ are, respectively, defined by
\begin{align*}
	P_f=\frac{f^{\prime\prime}(z)}{f^{\prime}(z)}\;\mbox{and}\;S_f=(P_f)^{\prime}(z)-\frac{1}{2}(P_f)^2(z)=\frac{f^{\prime\prime\prime}(z)}{f^{\prime\prime}(z)}-\frac{3}{2}\left(\frac{f^{\prime\prime}(z)}{f^{\prime}(z)}\right)^2.
\end{align*}
The pre-Schwarzian and Schwarzian norms are defined by
\begin{align*}
	||P_f||_{\Omega}= \sup_{z\in{\Omega}}|P_f|\eta^{-1}_{\Omega}\;\mbox{and}\;||S_f||_{\Omega}= \sup_{z\in{\Omega}}|S_f|\eta^{-2}_{\Omega}
\end{align*}
respectively, where $\eta_{\Omega}$ is the Poincare density. In particular, if ${\Omega}=\mathbb{D}$, then $||S_f||_{\Omega}$ and $||P_f||_{\Omega}$ are denoted by $||S_f||$ and $||P_f||$,  respectively.\vspace{2mm}

In the following, we discuss some properties of Schwarzian derivatives:
\begin{enumerate}
	\item[$\bullet$] If $\varphi$ is a locally univalent analytic function for which the composition $f\circ \varphi$ is defined, then 
	\begin{align*}
		S_{f\circ \varphi}(z)=S_{f}\circ \varphi(z)\left(\varphi^{\prime}(z)\right)+S_{\varphi}(z).
	\end{align*}
	\item[$\bullet$] The Schwarzian derivative is invariant under M\"obius transformation, \emph{i.e.,} $S_{T\circ\varphi}=S_f$ for any M\"obius transformation $T$ of the form 
	\begin{align*}
		T(z)=\frac{az+b}{cz+d},\; ad-bc\neq 0,\; a, b, c, d\in\mathbb{C}.
	\end{align*}
	\item[$\bullet$] It is easy to verify that $S_f(z)=0$ if, and only, if $f$ is a M\"obius transformation.\vspace{1mm}
	\item[$\bullet$] There is a classical relation between the Schwarzian derivative and second order linear differential equations. If $S_f=2p$ and $u=\left(f^{\prime}\right)^{-1/2}$, then 
	\begin{align*}
		u^{\prime\prime}+pu=0.
	\end{align*}
	Conversely, if $u_1$, $u_2$ are linearly independent solutions of this D.E. and $f=u_1/u_2$, then $S_f=2p$.  
\end{enumerate}
\;\;\;\;\;\;The pre-Schwarzian and Schwarzian derivatives are key tools in geometric function theory, particularly for characterizing Teichmüller space through embedding models. They also play a crucial role in studying the inner radius of univalency for planar domains and quasiconformal extensions \cite{Lehto-1987,Lehto-JAM-1979}. Their study dates back to Kummer (1836), who introduced the Schwarzian derivative in the context of hypergeometric $PDEs$. Since then, extensive research has explored their connections to univalent functions, leading to several sufficient conditions for univalency.
Let $\mathcal{A}$ denote the subclass of $\mathcal{H}$, a class of analytic functions on the unit disk $\mathbb{D}$,  consisting of functions $f$ with normalized conditions $f(0)=f^{\prime}(z)-1=0$. Thus, any function $f$ in $\mathcal{A}$ has the Taylor series expansion of the form 
\begin{align}\label{Eq-1.1}
	f(z)=z+\sum_{n=2}^{\infty}a_nz^n\;\; \mbox{for all}\;\; z\in\mathbb{D}.
\end{align}
\;\;\;\;\;Let $\mathcal{S}$ be the subclass of $\mathcal{A}$ consisting of univalent (that is, one-to-one) functions. A function $f\in \mathcal{A}$ is called starlike (with respect to the origin) if $f(\mathbb{D})$ is starlike with respect to the origin and convex if $f(\mathbb{D})$ is convex. The class of all univalent starlike (resp. convex) functions in $\mathcal{A}$ is denoted by $\mathcal{S}^*$ (resp. $\mathcal{C}$). However, it is well-known that a function $f\in\mathcal{S}^*$ (resp. $f\in\mathcal{C}$) if, and only if, 
\begin{align*}
	{\rm Re}\left(\frac{zf^{\prime}(z)}{f(z)}\right)>0\; \left(\mbox{resp.}\; {\rm Re}\left(1+\frac{zf^{\prime\prime}(z)}{f^{\prime}(z)}\right)>0\right),\; z\in\mathbb{D}.
\end{align*}
The following observations are important for starlike and convex functions.
\begin{enumerate}
	\item[$\bullet$] The characterizations of starlikeness or convexity are sufficient but not necessary for univalency.\vspace{2mm}
	
	\item[$\bullet$] {\bf Nehari's criteria.} Nehari developed univalency involving Schwarzian derivative, where sufficient condition is almost necessary in the sense that scalar terms vary.
\end{enumerate}
The next two results provide necessary and sufficient criteria for a function to be univalent.
\begin{thmA}\emph{(Kraus-Nehari's Theorem) (Necessary condition)}
	Let $f$ be a univalent function. Then $f$ satisfies
	\begin{align*}
		|S_f(z)|\leq\frac{6}{(1-|z|^2)^2}\; \mbox{for}\; z\in\mathbb{D}.
	\end{align*}
	Moreover, the constant $6$ cannot be replaced by a smaller one.
\end{thmA}
\begin{thmB}\emph{(Nehari's Theorem) (Sufficient condition)}
	Let $f$ be a locally univalent function. If $f$ satisfies
	\begin{align*}
		|S_f(z)|\leq\frac{2}{(1-|z|^2)^2}\; \mbox{for}\; z\in\mathbb{D},
	\end{align*}
	then $f$ is univalent in $\mathbb{D}$. Moreover, the constant $2$ cannot be replaced by a larger one.
\end{thmB}
This was Nehari's motivation to study about the Schawarzian derivatives as well as Schawarzian norm. It is well-known that the pre-Schwarzian norm $||Pf||\leq 6$ holds for the univalent analytic function $f$ is defined in $\mathbb{D}$. In $1972$, Becker \cite{Becker-JRAM-1983} used the pre-Schwarzian derivative to obtain the sufficient condition that the function in $\mathbb{D}$ is univalent, in other words, if $||Pf||\leq1$, then the function $f$ is univalent in $\mathbb{D}$. In 1976, Yamashita \cite{Yamashita-MM-1976} proved that $||Pf||$ is finite if, and only if, $f$ is uniformly locally univalent in $\mathbb{D}$, \emph{i.e.}, there exists a constant $\rho$ such that $f$ is univalent on the hyperbolic disk $|(z-a)/(1-\bar{a}z)|<\tanh\rho$ of radius $\rho$ for every $a\in\mathbb{D}$. Sugawa \cite{Sugawa-AUMCDS-1996} studied and established the norm of the pre-Schwarzian derivative of the strongly starlike functions of order $\alpha\; (0<\alpha\leq1)$. Yamashita\cite{Yamashita-HMJ-1999} generalized sugawa's results by a general class named Gelfer-starlike of exponential order $\alpha (\alpha>0)$ and the Gelfer-close-to-convex of exponential order $(\alpha,\beta)$ ($\alpha>0$, $\beta>0$). These Gelfer classes also contain the classical starlike, convex, close-to-convex all of order $\alpha$ ($0\leq\alpha<1$), which are denote by $\mathcal{S^*(\alpha)}$, $\mathcal{C(\alpha)}$, $\mathcal{K(\alpha)}$ respectively, and so on.\vspace{2mm}

Here, we recall that a function $f\in\mathcal{A}$ is called close-to-convex if $f(\mathbb{D})$ in $\mathbb{C}$ is the union of closed half lines with pairwise disjoint interiors. However, in \cite{Okuyama-CVTA-2000}, Okuyama  studied the subclass of $\alpha$-spirallike functions of order ($-\pi/2<\alpha<\pi/2$), and later a general class call $\alpha$-spirallike functions of order $\rho$ $(0\leq\rho<1)$ considered by Aghalary and Orouji \cite{Aghalary-Orouji-COAT-2014}. Recently, Ali and Pal \cite{Ali-Pal-MM-2023} studied the sharp estimate of the pre-Schwarzian norm for the Janowski starlike functions. Other subclasses have also been widely studied, such as meromorphic function exterior of the unit disk \cite{Ponnusamy-Sugawa-JKMS-2008}, subclass of strong starlike function \cite{Ponnusamy-Sahoo-M-2008}, uniformly convex and uniformly starlike function \cite{Kanas-AMC-2009} and bi-univalent function \cite{Rahmatan-Najafzadeh-Ebadian-BIMS-2017}. For the pre-Schwarzian norm estimates of other function forms such as convolution operator and integral operator, we refer to the articles \cite{Choi-Kim-Ponnusamy-Sugawa-JMAP-2005,Kim-Sugawa-PEMS-2006,Parvatham-Ponnusamy-Sahoo-HMJ-2008,Ponnusamy-Sahoo-JMAA-2008} and references therein. The pioneering work on the bound \( ||Sf|| \leq 6 \) for a univalent function \( f \in \mathcal{A} \) was first introduced by Kraus \cite{Kraus-1932} and later revisited by Nehari \cite{Nehari-BAMS-1949}. In the same paper, Nehari also proved that if $||Sf||\leq2$, then the function $f$ is univalent in $\mathbb{D}$.\vspace{2mm}

The Schwarzian norm plays a significant role in the theory of quasiconformal mappings and Teichm\"uller space (see \cite{Lehto-1987}). A mapping \( f : \widehat{\mathbb{C}} \to \widehat{\mathbb{C}} \) of the Riemann sphere \( \widehat{\mathbb{C}} := \mathbb{C} \cup \{\infty\} \) is said to be a \( k \)-quasiconformal (\( 0 \leq k < 1 \)) mapping if it is a sense-preserving homeomorphism of \( \widehat{\mathbb{C}} \) and has locally integrable partial derivatives on \( \mathbb{C} \setminus \{f^{-1}(\infty)\} \), satisfying \( |f_{\bar{z}}| \leq k |f_z| \) almost everywhere.  On the other hand, Teichm\"uller space \( \mathcal{T} \) can be identified with the set of Schwarzian derivatives of analytic and univalent functions on \( \mathbb{D} \) that have quasiconformal extensions to \( \widehat{\mathbb{C}} \). It is known that \( \mathcal{T} \) is a bounded domain in the Banach space of analytic functions on \( \mathbb{D} \) with a finite hyperbolic sup-norm (see \cite{Lehto-1987}).  \vspace{2mm}

The Schwarzian derivative and quasiconformal mappings are connected through key results presented below.
\begin{thmC}\cite{Ahlfrors-Weill-PAMS-2012,Kühnau-MN-1971}
	If $f$ extends to a $k$-quasiconformal $(0\leq k<1)$ mapping of the Riemann share $\widehat{\mathbb{C}}$, then $||S_f||\leq 6k$. Conversely, if $||S_f||\leq 2k$, then $f$ extends to a $k$-quasiconformal mapping of the Riemann sphere $\widehat{\mathbb{C}}$.
\end{thmC}
Regarding to the estimates of the Schwarzian norm for the subclasses of univalent functions \emph{i.e.,} of functions $f$ that satisfy:
\begin{align*}
	\bigg|\arg\left(\frac{zf^{\prime}(z)}{f(z)}\right)\bigg|<\alpha\frac{\pi}{2},\; z\in\mathbb{D},
\end{align*}
where $0\leq \alpha<1$. In $1996$, Suita \cite{Suita-JHUED-1996} studied the class $\mathcal{C(\alpha)}$, $0\leq \alpha<1$ and using the integral representation of functions in $\mathcal{C}(\alpha)$ proved that the Schwarzian norm satisfies the sharp inequality 
\begin{align*}
	||S_f||\leq\begin{cases}
		2,\;\;\;\;\;\;\;\;\;\;\;\;\;\;\;\;\; \mbox{if}\; 0\leq \alpha\leq 1/2,\\
		8\alpha(1-\alpha),\;\;\;\;\mbox{if}\; 1/2<\alpha<1.
	\end{cases}
\end{align*}
\;\;\;\;For a constant $\beta\in (-\pi/2, \pi/2)$, a function $f\in\mathcal{A}$ is called $\beta$-spiral like if $f$ is univalent on $\mathbb{D}$ and for any $z\in\mathbb{D}$, the $\beta$-logarithmic spiral $\{f(z)\exp\left(-e^{i\beta}t\right);\; t\geq 0\}$ is contained in $f(\mathbb{D})$. It is equivalent to the condition that ${\rm Re} \left(e^{-i\beta}zf^{\prime}(z)/f(z)\right)>0$ in $\mathbb{D}$ and we denote by $\mathcal{SP}(\beta)$, the set of all $\beta$-spiral like functions. Okuyama \cite{Okuyama-CVTA-2000} give the best possible estimate of the norm of pre-Schwarzian derivatives for the class $\mathcal{SP}(\beta)$.\vspace{2mm}

A function $f\in\mathcal{A}$ is said to be uniformly convex function if every circular arc (positively oriented) of the form $\{z\in\mathbb{D} : |z-\eta|=r\}$, $\eta\in\mathbb{D}$, $0<r<|\eta|+1$ is mapped by $f$ univalently onto a convex arc. The class of all uniformly convex functions is denoted by $\mathcal{UCV}$. In particular, $\mathcal{UCV}\subset \mathcal{K}$. It is well-known that (see \cite{Goodman-APM-1991}) a function $f\in\mathcal{A}$ is uniformly convex if, and only if, 
\begin{align*}
	{\rm Re}\left(1+\frac{z f^{\prime\prime}(z)}{f^{\prime}(z)}\right)>\bigg|\frac{z f^{\prime\prime}(z)}{f^{\prime}(z)}\bigg|^2\; \mbox{for}\; z\in\mathbb{D}.
\end{align*}
\;\;\;\;In \cite{Kanas-Sugawa-APM-2011}, Kanas and Sugawa established that the Schwarzian norm satisfies \( ||S_f|| \leq 8/\pi^2 \) for all \( f \in \mathcal{UCV} \), with the bound being sharp.  Recently, Schwarzian norm estimates for other subclasses of univalent functions have been gradually studied by many people, such as concave function class \cite{Bhowmik-Wriths-CM-2012}, Robertson class \cite{Ali-Pal-BDS-2023} and other univalent analytic subclasses [6]. Therefore, by using the pre-Schwarzian  and Schwarzian norms to study the univalence and quasiconformal extension problems of analytic function arouse a new wave of research interest.\vspace{2mm}

	A domain $\Omega$ containing the origin is called $\alpha$-spirallike if for each point $\omega_0$ in $\Omega$ the arc of the $\alpha$-spiral from the origin to the point $\omega_0$ entirely lies in $\Omega$. A function $f\in\mathcal{A}$ is said to be an $\alpha$-spirallike if
	\begin{align*}
		{\rm Re}\left(e^{i\alpha}\frac{z f^{\prime}(z)}{f(z)}\right)>0\;\mbox{for}\; z\in\mathbb{D},
	\end{align*}
	where $|\alpha|<\pi/2$. In $1933$, $\check{S}$pa$\check{c}$ek (see \cite{Spacek-CPMF-1933}) introduced and studied the class of $\alpha$-spirallike functions and this class is denoted by $\mathcal{SP}(\alpha)$. Later on, Robertson \cite{Robertson-MMJ-1969} introduced a new class of functions, denoted by $\mathcal{S}_\alpha$, in connection with $\alpha$-spirallike functions. A function $f\in\mathcal{A}$ is in the class $\mathcal{S}_\alpha$ if and only if
	\begin{align*}
		{\rm Re}\left(e^{i\alpha}\left(1+\frac{z f^{\prime\prime}(z)}{f^{\prime}(z)}\right)\right)>0\;\mbox{for}\; z\in\mathbb{D}.
	\end{align*}
	
	\;\;\;\; Let us introduce one of the most important and useful tool known as differential subordination technique. In geometric function theory, many problems can be solved in a simple and sharp manner with the help of differential subordination. A function $f\in\mathcal{H}$ is said to be subordinate to another function $g\in\mathcal{H}$ if there exists an analytic function $\omega : \mathbb{D}\to\mathbb{D}$ with $\omega(0)=0$ such that $f(z)=g(\omega(z))$ and it is denoted by $f\prec g$. Moreover, when $g$ is univalent, then $f\prec g$ if, and only if, $f(0)=g(0)$ and $f(\mathbb{D}\subset g(\mathbb{D})$. In terms of subordination, the class $\mathcal{SP_\alpha}(\beta)$ can be defined in the following form
	\begin{align}\label{Eq-1.2}
		f\in\mathcal{SP_\alpha}(\beta)\iff1+\frac{zf^{\prime\prime}(z)}{f^{\prime}(z)}\prec\frac{1+Az}{1-z}
	\end{align}
	where, $A=e^{-\iota\alpha}(e^{-\iota\alpha}-2\beta\cos\alpha)$.\vspace*{2mm}
	
	In this paper, we establish various geometric properties of functions in the class $\mathcal{SP_\alpha}(\beta)$, including growth and distortion theorems, and examine the sharpness of these results. Furthermore, we determine the sharp estimates for the Schwarzian and pre-Schwarzian norms of functions in this class. Finally, addressing a significant problem in geometric function theory, we determine the sharp radius of concavity and radius of convexity for the class $\mathcal{SP_\alpha}(\beta)$. The manuscript is organized into two sections. In Section \ref{Sec-2}, we present all relevant results, including bounds for the Schwarzian and pre-Schwarzian derivatives, norms, the growth-distortion theorem, and the sharp bounds for the class $\mathcal{SP_\alpha}(\beta)$. In Section \ref{Sec-3}, we determine the radius of concavity and the radius of convexity for functions in the class $\mathcal{SP_\alpha}(\beta)$. The detailed proofs of the main results are discussed in each respective section.
	
\section{\bf{Pre-Schwarzian and Schwarzian norm estimates for Robertson class}}\label{Sec-2} In \cite{Chuaqui-Duren-Osgood-AASFM-2011}, Chuaqui \emph{et. al.} proved a result by applying the Schwarz-Pick lemma and the fact that the expression $1+z(f^{\prime\prime}/f^{\prime})(z)$ is subordinate to the half-plan mapping $\ell(z)=(1+z)/(1-z)$, which is 
\begin{align}\label{Eq-1.3}
	1+\frac{zf^{\prime\prime}(z)}{f^{\prime}(z)}=\ell(w(z))=\frac{1+w(z)}{1-w(z)}
\end{align}
for some function $w : \mathbb{D}\to\mathbb{D}$ holomorphic and such that $w(0)=0$. \vspace{2mm}

The expression which is defined in \eqref{Eq-1.3} allowed us to obtain other characterizations for the convex functions: 
\begin{align}\label{Eq-22.3}
	f\in\mathcal{C}\; \mbox{if, and only if,}\; {\rm Re}\left(1+\frac{zf^{\prime\prime}(z)}{f^{\prime}(z)} \right)\geq \frac{1}{4}\left( 1-|z|^2\right)\bigg| \frac{f^{\prime\prime}(z)}{f^{\prime}(z)}  \bigg|^2,
\end{align}
and 
\begin{align}\label{Eq-22.4}
	f\in\mathcal{C}\; \mbox{if, and only if,}\; \bigg|\left(1-|z|^2\right)\frac{f^{\prime\prime}(z)}{f^{\prime}(z)} -2\bar{z} \bigg|\leq 2,
\end{align}
for all $z\in\mathbb{D}$.\vspace{2mm}

In this section, inspired the article \cite{Wang-Li-Fan-MM-2024,Carrasco-Hernández-AMP-2023}, we firstly give the equivalent characterization of the class $\mathcal{SP}^{0}_\alpha(\beta)$ (Robertson class), next we present the distortion and growth theorem, and then we derive the results of pre-Schwarzian and Schwarzian norms for the class of $\mathcal{SP}^{0}_\alpha(\beta)$ (Robertson class) in terms of the value of $f^{\prime\prime}(0)$. We define $G_j(\alpha, \beta)$ $(j=1, 2)$ as follows:
\begin{align*}
	\begin{cases}
		G_1(\alpha, \beta):=\dfrac{\left(e^{-i\alpha}\left(e^{-i\alpha}-2\beta\cos\alpha\right)+1\right)}{2},\vspace{2mm}\\
			G_2(\alpha, \beta):=\dfrac{1-|z|^2}{\left(e^{-i\alpha}\left(e^{-i\alpha}-2\beta\cos\alpha\right)+1\right)}
	\end{cases}
\end{align*}
and obtain the result.
\begin{theorem}\label{Th-2.1} For $-\pi/2<\alpha<\pi/2$ and $0\leq\beta<1$ the following are equivalent:
	\begin{enumerate}
		\item[\emph{(iii)}] $f\in\mathcal{SP}^{0}_\alpha(\beta)$\vspace{2mm}
		
		\item[\emph{(ii)}] 
		\begin{align}\label{Eq-2.1A}
			{\rm Re}&\left(1+\frac{1}{2}\left((e^{2i\alpha}-2\beta e^{i\alpha}\cos\alpha)+1\right)\frac{zf^{\prime\prime}(z)}{f^{\prime}(z)}\right)\\&\geq1-(1-\beta)^2\cos^2\alpha+\left(\frac{1-|z|^2}{4}\right)\bigg|\frac{zf^{\prime\prime}(z)}{f^{\prime}(z)}\bigg|^2\nonumber
		\end{align}
		
		\item[\emph{(iii)}]
		\begin{align}\label{Eq-2.2A}
		\bigg|	(1-|z|^2)\left(\frac{f^{\prime\prime}(z)}{f^{\prime}(z)}\right)-2(1-\beta)\cos\alpha\bar{z}\bigg|\leq (1-\beta)\cos\alpha.
		\end{align}
	\end{enumerate}
	The inequalities \emph{(ii)} and \emph{(iii)} both are sharp for the function
	\begin{align}\label{Eq-2.3A}
		f^{\prime}_{\alpha,\beta}(z)=\frac{1}{(1-z)^{(1-\beta)\cos\alpha}}\; \mbox{for}\; z\in\mathbb{D}\; \mbox{with}\;\beta\in[0,1).
	\end{align}
\end{theorem}
\begin{proof}[\bf Proof of Theorem \ref{Th-2.1}]
For $-\pi/2<\alpha<\pi/2$ and $0\leq \beta<1$ let  $f\in\mathcal{SP}^{0}_\alpha(\beta)$ be of the form \eqref{Eq-1.1}. Then from \eqref{Eq-1.2}, we have
\begin{align}\label{Eq-2.1}
	1+\frac{zf^{\prime\prime}(z)} {f^{\prime}(z)}\prec\frac{1+Az}{1-z},\;\;\mbox{where}\;A=e^{-i\alpha}\left(e^{-i\alpha}-2\beta\cos\alpha\right).
\end{align}
Consequently, there exists an analytic function $\omega: \mathbb{D}\rightarrow\mathbb{D}$ with $\omega(0)=0$ such that
     \begin{align*}
			1+\frac{zf^{\prime\prime}(z)} {f^{\prime}(z)}=\frac{1+A\omega(z)}{1-\omega(z)}.
			\end{align*}
 Let $\omega(z)=z\phi(z)$ for some analytic function $\phi$ that satisfy $\phi(\mathbb{D})\subseteq\mathbb{D}$. From \eqref{Eq-2.1}, we have
	\begin{align}\label{Eq-2.3}
		\frac{f^{\prime\prime}(z)}{f^{\prime}(z)}=\frac{2G_1(\alpha, \beta)\omega(z)}{z(1-\omega(z))}
	\end{align}
	which can be written as
	\begin{align*}
		\frac{f^{\prime\prime}(z)}{f^{\prime}(z)}=\frac{2G_1(\alpha, \beta)\;\phi(z)}{(1-z\phi(z))}
	\end{align*}
	Thus, we see that
	\begin{align}\label{Eq-2.4}
		\phi(z)=\frac{\frac{f^{\prime\prime}(z)}{f^{\prime}(z)}}{2G_1(\alpha, \beta)+\frac{zf^{\prime\prime}(z)}{f^{\prime}(z)}}.
	\end{align}
	Since $|\phi(z)|^2\leq1$, an easy computation shows that
	\begin{align}\label{Eq-4.5}
		\bigg|\frac{f^{\prime\prime}(z)}{f^{\prime}(z)}\bigg|^2\leq\left(2G_1(\alpha, \beta)+\frac{zf^{\prime\prime}(z)}{f^{\prime}(z)}\right)\overline{\left(2G_1(\alpha, \beta)+\frac{zf^{\prime\prime}(z)}{f^{\prime}(z)}\right)}.
	\end{align}
	A simple computation shows that
	\begin{align}\label{Eq-2.6}
		(1-|z|^2)\bigg|\frac{f^{\prime\prime}(z)}{f^{\prime}(z)}\bigg|^2\leq(1-\beta)^2\cos^2\alpha+4{\rm Re}\left(\frac{1}{2}\left((e^{2i\alpha}-2\beta e^{i\alpha}\cos\alpha)+1\right)\frac{zf^{\prime\prime}(z)}{f^{\prime}(z)}\right)
	\end{align}
	which is equivalent to
	\begin{align*}
		{\rm Re}&\left(\frac{1}{2}\left((e^{2i\alpha}-2\beta e^{i\alpha}\cos\alpha)+1\right)\frac{zf^{\prime\prime}(z)}{f^{\prime}(z)}\right)\\&\geq-(1-\beta)^2\cos^2\alpha+\left(\frac{1-|z|^2}{4}\right)\bigg|\frac{zf^{\prime\prime}(z)}{f^{\prime}(z)}\bigg|^2
	\end{align*}
	We rewrite the last expression as
	\begin{align}\label{Eq-2.7}
		{\rm Re}&\left(1+\frac{1}{2}\left((e^{2i\alpha}-2\beta e^{i\alpha}\cos\alpha)+1\right)\frac{zf^{\prime\prime}(z)}{f^{\prime}(z)}\right)\\&\geq 1-(1-\beta)^2\cos^2\alpha+\left(\frac{1-|z|^2}{4}\right)\bigg|\frac{zf^{\prime\prime}(z)}{f^{\prime}(z)}\bigg|^2.\nonumber
	\end{align}
	Multiplying both sides of equation $\eqref{Eq-2.6}$ by $(1-|z|^2)$, we obtain
	\begin{align}
		&(1-|z|^2)^2\bigg|\frac{f^{\prime\prime}(z)}{f^{\prime}(z)}\bigg|^2\\&\leq(1-|z|^2)(1-\beta)^2\cos^2\alpha+4(1-|z|^2){\rm Re}\left(\frac{1}{2}\left((e^{2i\alpha}-2\beta e^{i\alpha}\cos\alpha)+1\right)\frac{zf^{\prime\prime}(z)}{f^{\prime}(z)}\right)\nonumber
	\end{align}
	which implies that 
	\begin{align*}
		(1-|z|^2)^2&\bigg|\frac{f^{\prime\prime}(z)}{f^{\prime}(z)}\bigg|^2-4(1-|z|^2){\rm Re}\left(\frac{1}{2}\left((e^{2i\alpha}-2\beta e^{i\alpha}\cos\alpha)+1\right)\frac{zf^{\prime\prime}(z)}{f^{\prime}(z)}\right)\\&\quad+|z|^2(1-\beta)^2\cos^2\alpha\leq(1-\beta)^2\cos^2\alpha.
	\end{align*}
	Thus, we have
	\begin{align}\label{Eq-2.8}
	\bigg|	(1-|z|^2)\left(\frac{f^{\prime\prime}(z)}{f^{\prime}(z)}\right)-2(1-\beta)\cos\alpha\bar{z}\bigg|\leq (1-\beta)\cos\alpha.
	\end{align}
	This completes the proof.
\end{proof}

\begin{example}
	For the sharpness of the inequalities \eqref{Eq-2.1A} and \eqref{Eq-2.2A}, we consider the function defined in \eqref{Eq-2.3A} with $\beta=0$ as 
	\begin{align*}
		f^{\prime}_{0}(z)=\frac{1}{(1-z)}
	\end{align*}
	A simple computation using \eqref{Eq-2.3A} shows that
	\begin{align*}
		1+\frac{zf_0^{\prime\prime}(z)}{f_0^{\prime}(z)}=1+\frac{z}{1-z}.
	\end{align*}
	Moreover, it is easy to see that
	\begin{align*}
		{\rm 	Re}\left(1+\frac{zf_{0}^{\prime\prime}(z)}{f_{0}^{\prime}(z)}\right)>0,
	\end{align*}
	hence, it is clear that $f_{0}\in\mathcal{SP}^{0}_0(0)$.\vspace{2mm}
	
	To show the inequality \eqref{Eq-2.13A} of Corollary \ref{Cor-2.2} is sharp, we consider $z=r<1$ and establish that 
	\begin{align*}
		\bigg|	(1-|z|^2)\left(\frac{f_{0}^{\prime\prime}(z)}{f_{0}^{\prime}(z)}\right)-\bar{z}\bigg|=\bigg|	(1-r^2)\left(\frac{1}{1-r}\right)-r\bigg|=|1-r+r|=1.
	\end{align*}
	
	\noindent To show the inequality \eqref{Eq-2.13A} in Corollary \ref{Cor-2.1} is sharp, we see from \eqref{Eq-2.4} (Proof of Theorem \ref{Th-2.1}) that 
	\begin{align}\label{Eq-22.44}
		\phi(z)=\frac{\frac{f_{0}^{\prime\prime}(z)}{f_{0}^{\prime}(z)}}{\frac{zf_{0}^{\prime\prime}(z)}{f_{0}^{\prime}(z)}+2}=1.
	\end{align}
	Thus, it is clear that $|\phi(z)|^2=1$ which further leads to 
	\begin{align*}
		{\rm Re  }\left(1+\frac{zf_{1/2}^{\prime\prime}(z)}{f_{1/2}^{\prime}(z)}\right)=\left(\frac{1-|z|^2}{4}\right)\bigg|\frac{zf_1^{\prime\prime}(z)}{f_1^{\prime}(z)}\bigg|^2.
	\end{align*}
\end{example}
We have the following immediate results form Theorem \ref{Th-2.1}.
\begin{corollary}\label{Cor-2.1}
If $\alpha=\beta=0$, $f\in \mathcal{SP}^{0}_0(0)\subset\mathcal{C}$, then from  \eqref{Eq-2.7} we have 
	\begin{align}\label{Eq-2.13A}
	{\rm Re}\left(1+\frac{zf^{\prime\prime}(z)}{f^{\prime}(z)}\right)\geq\frac{1}{4}(1-|z|^2)\bigg|\frac{zf^{\prime\prime}(z)}{f^{\prime}(z)}\bigg|^2.
	\end{align}
\end{corollary}
We obtain the following corollary for the class $\mathcal{C}$.
\begin{corollary}\label{Cor-2.2}
	If $\alpha=\beta=0$, $f\in \mathcal{SP}^{0}_0(0)\subset\mathcal{C}$, then from  \eqref{Eq-2.8} we have 
	\begin{align*}
		\bigg|(1-|z|^2)\frac{f^{\prime\prime}(z)}{f^{\prime}(z)}-2\bar{z}\bigg|\leq1.
	\end{align*}
\end{corollary}
\begin{remark}
	The inequality 
	\begin{align*}
		\bigg|	(1-|z|^2)\left(\frac{f^{\prime\prime}(z)}{f^{\prime}(z)}\right)-2(1-\beta)\cos\alpha\bar{z}\bigg|\leq (1-\beta)\cos\alpha.
	\end{align*}
	is instrumental in the definition or characterization of the radius of concavity.
\end{remark}
In the next result, we establish the distortion theorem and growth theorem for the functions in the class $\mathcal{SP_\alpha}^0(\beta)=\{f\in\mathcal{SP_\alpha}(\beta):f^{\prime\prime}(0)=0\}$.
\begin{theorem}\label{Th-2.2}
	For $-\pi/2<\alpha<\pi/2$ and $0\leq\beta<1$, let $f\in \mathcal{SP}^{0}_\alpha(\beta)$ be of the form \eqref{Eq-1.1} for all $z\in\mathbb{D}$, then the inequality
	\begin{align*}
		\frac{1}{(1+|z|^2)^{(1-\beta)\cos\alpha}}\leq|f^{\prime}(z)|\leq\frac{1}{(1-|z|^2)^{(1-\beta)\cos\alpha}}
	\end{align*}
	and
	\begin{align*}
			\int_{0}^{|z|}\frac{1}{(1+\xi^2)^{(1-\beta)\cos\alpha}} d|\xi|\leq	|f(z)|\leq\int_{0}^{|z|}\frac{1}{(1-\xi^2)^{(1-\beta)\cos\alpha}} d|\xi|.
	\end{align*}
	All of these estimates are sharp. Equality holds at a given point other than $0$ for 
	\begin{align*}
		f(z)=\int_{0}^{|z|}\frac{1}{(1-\lambda \zeta^2)^{(1-\beta)\cos\alpha}}d|\zeta|
	\end{align*} for some $\lambda\in\mathbb{C}$ and $|\lambda|=1$.
\end{theorem}
We have the following immediate result for a subclass of $\mathcal{C}$ of convex functions.
\begin{corollary}
	For $\alpha=\beta=0$, let $f\in \mathcal{SP}^{0}_0(0)\subset\mathcal{C}$ be of the form \eqref{Eq-1.1}, then the inequality
	\begin{align*}
		\frac{1}{(1+|z|^2)}\leq|f^{\prime}(z)|\leq\frac{1}{(1-|z|^2)}
	\end{align*}
	and
	\begin{align*}
		\int_{0}^{|z|}\frac{1}{(1+\xi^2)} d|\xi|\leq	|f(z)|\leq\int_{0}^{|z|}\frac{1}{(1-\xi^2)} d|\xi|.
	\end{align*}
	All of these estimates are sharp. Equality holds at a given point other than $0$ for 
	\begin{align*}
		f(z)=\int_{0}^{|z|}\frac{1}{(1-\lambda \zeta^2)}d|\zeta|
	\end{align*}
	 for some $\lambda\in\mathbb{C}$ and $|\lambda|=1$.
\end{corollary}

\begin{proof}[\bf Proof of Theorem \ref{Th-2.2}]
	Let $f\in\mathcal{SP}^{0}_\alpha(\beta)$ be of the form \eqref{Eq-1.1} and from \eqref{Eq-2.4}, we obtain $\phi(0)=0$. Then  by using the Schwarz lemma, we get
	\begin{align}
		\left|\dfrac{\dfrac{f^{\prime\prime}(z)}{f^{\prime}(z)}}{2G_1(\alpha, \beta)+\dfrac{zf^{\prime\prime}(z)}{f^{\prime}(z)}}\right|^2\leq|z|^2
	\end{align}
	which implies that
	\begin{align*}
	\bigg|\frac{f^{\prime\prime}(z)}{f^{\prime}(z)}\bigg|^2&\leq4|z|^2(1-\beta)^2\cos^2\alpha+4|z|^2{\rm Re}\left(\frac{1}{2}\left((e^{2i\alpha}-2\beta e^{i\alpha}\cos\alpha)+1\right)\frac{zf^{\prime\prime}(z)}{f^{\prime}(z)}\right)\\&\quad+|z|^4\bigg|\frac{f^{\prime\prime}(z)}{f^{\prime}(z)}\bigg|^2.
	\end{align*}
	Thus, we have
	\begin{align}\label{Eq-2.10}
		&(1-|z|^4)\bigg|\frac{f^{\prime\prime}(z)}{f^{\prime}(z)}\bigg|^2\\&\quad\leq 4|z|^2(1-\beta)^2\cos^2\alpha+4|z|^2{\rm Re}\left(\frac{1}{2}\left((e^{2i\alpha}-2\beta e^{i\alpha}\cos\alpha)+1\right)\frac{zf^{\prime\prime}(z)}{f^{\prime}(z)}\right).\nonumber
	\end{align}
	Multiplying both sides of \eqref{Eq-2.10} by $(1-|z|^4)$, we obtain 
	\begin{align*}
		(1-|z|^4)^2&\bigg|\frac{f^{\prime\prime}(z)}{f^{\prime}(z)}\bigg|^2-4|z|^2(1-|z|^4){\rm Re}\left(\frac{1}{2}\left((e^{2i\alpha}-2\beta e^{i\alpha}\cos\alpha)+1\right)\frac{zf^{\prime\prime}(z)}{f^{\prime}(z)}\right)\\&\leq4|z|^2(1-|z|^4)(1-\beta)^2\cos^2\alpha.
	\end{align*}
	Adding  $\left(2(1-\beta)\cos\alpha|z|^2\bar{|z|}\right)^2$ both side of the above inequality, we get
	\begin{align}
			(1-|z|^4)^2&\bigg|\frac{f^{\prime\prime}(z)}{f^{\prime}(z)}\bigg|^2-4|z|^2(1-|z|^4){\rm Re}\left(\frac{1}{2}\left((e^{2i\alpha}-2\beta e^{i\alpha}\cos\alpha)+1\right)\frac{zf^{\prime\prime}(z)}{f^{\prime}(z)}\right)\\&\quad\nonumber+4(1-\beta)^2\cos^2\alpha|z|^4|\bar{z}|^2\\&\leq4|z|^2(1-|z|^4)(1-\beta)^2\cos^2\alpha+4(1-\beta)^2\cos^2\alpha|z|^4|\bar{z}|^2\nonumber.
	\end{align}
	Multiplying both side by $|z|$, then by simple calculation
	\begin{align}
		\bigg|(1-|z|^4)\frac{zf^{\prime\prime}(z)}{f^{\prime}(z)}-2(1-\beta)\cos\alpha|z|^4\bigg|\leq 2(1-\beta)\cos\alpha|z|^2
	\end{align}
	which implies 
	\begin{align}
		\frac{-2(1-\beta)\cos\alpha|z|^2}{1+|z|^2}\leq{\rm Re}\left(\frac{zf^{\prime\prime}(z)}{f^{\prime}(z)}\right)\leq\frac{2(1-\beta)\cos\alpha|z|^2}{1-|z|^2}.
	\end{align}
	Let $z=re^{i\theta}$. Then, we obtain
	\begin{align*}
		\frac{-2r(1-\beta)\cos\alpha}{1+r^2}\leq\dfrac{\partial}{\partial r}\left(\log|f^{\prime}(re^{i\theta})|\right)\leq\frac{2r(1-\beta)\cos\alpha}{1-r^2}.
	\end{align*}
	\noindent{\bf Case A.} When $\alpha=0$ and $0\leq\beta<1$ if we integrate respect to r, we obtain\begin{align}
		\frac{1}{(1+|z|^2)^{1-\beta}}\leq|f^{\prime}(z)|\leq\frac{1}{(1-|z|^2)^{1-\beta}}
	\end{align}
	\noindent{\bf Case B.} When $\alpha\neq0$, if we integrate respect to r, we obtain
	\begin{align}
		\frac{1}{(1+|z|^2)^{(1-\beta)\cos\alpha}}\leq|f^{\prime}(z)|\leq\frac{1}{(1-|z|^2)^{(1-\beta)\cos\alpha}}.
	\end{align}
	Next, for the growth part of the theorem, from the upper bound it follows that
	\begin{align}
		|f^{\prime}(re^{i\theta})|=\bigg|\int_{0}^{r}f^{\prime}(re^{i\theta})e^{i\theta} dt\bigg|\leq\int_{0}^{r}|f^{\prime}(re^{i\theta})| dt\leq\int_{0}^{r}\frac{1}{(1-t^2)^{(1-\beta)\cos\alpha}} dt
	\end{align}
	which implies 
	\begin{align}
		|f(z)|\leq\int_{0}^{|z|}\frac{1}{(1-\xi^2)^{(1-\beta)\cos\alpha}} d|\xi|
	\end{align}
	for all $z\in\mathbb{D}$. It is well-known that if $f(z_0)$ is a point of minimum modulus on the image of the circle $|z|=r$ and $\gamma=f^{-1}(\Gamma)$, where $\Gamma$ is the line segment from $0$ to $f(z_0)$, then 
	\begin{align}
			|f(z)|\geq	|f(z_0)|\geq\int_{0}^{|z|}\frac{1}{(1+\xi^2)^{(1-\beta)\cos\alpha}} d|\xi|.
	\end{align}
Thus, the inequalities are established.
\end{proof}
Now, we will find the sharp bound of the pre-Schwarzian and Schwarzian norms for functions in the class $\mathcal{SP_\alpha}^0(\beta)=\{f\in\mathcal{SP_\alpha}(\beta):f^{\prime\prime}(0)=0\}$. The following lemma will play a key role to prove the result.
\begin{lemA}\cite{Carrasco-Hernández-AMP-2023}
	If $\phi(z):\mathbb{D}\rightarrow\mathbb{D}$ be analytic function, then 
	\begin{align}
		\frac{|\phi(z)|^2}{1-|\phi(z)|^2}\leq\frac{(\phi(0)+|z|)^2}{(1-|\phi(0)|)^2(1-|z|^2)|)}
	\end{align}
\end{lemA}
We obtain the following result establishing a sharp bound of the pre-Schwarzian norm for $f\in\mathcal{SP}^{0}_\alpha(\beta)$.
\begin{theorem}\label{Th-2.3}
For $0\leq\beta<1$ and $-\pi/2<\alpha<\pi/2$, let $f\in\mathcal{SP}^{0}_{\alpha}(\beta)$ be of the form \eqref{Eq-1.1} for all  $z\in\mathbb{D}$, then the pre-Schwarzian norm satisfies the inequality
\begin{align}
||Pf||\leq 2(1-\beta)\cos\alpha.
\end{align}
The inequality is sharp.
\end{theorem}
\begin{proof}[\bf Proof of Theorem \ref{Th-2.3}]
	Since $\phi(z)=z\xi(z)$, with $|\xi(z)|<1$, then in \eqref{Eq-2.3} we obtain
	\begin{align*}
		\sup_{z\in\mathbb{D}}(1-|z|^2)\bigg|\frac{f^{\prime\prime}(z)}{f^{\prime}(z)}\bigg|&\leq\sup_{z\in\mathbb{D}}(1-|z|^2)\frac{2(1-\beta)\cos\alpha|z\xi(z)|}{1-|z|^2|\xi(z)|}\\&\leq 2(1-\beta)\cos\alpha \sup_{0\leq r\leq 1}\frac{r(1-r^2)}{(1-r^2)}\\&=2(1-\beta)\cos\alpha.
	\end{align*}
	The extremal function is given by
	\begin{align}
		f^*(z)=	\int_{0}^{z}\frac{1}{(1-\xi^2)^{(1-\beta)\cos\alpha}} d\xi.
	\end{align}
		It can be easily shown that $||Pf^*||=2(1-\beta)\cos\alpha.$ This completes the proof.	
\end{proof}
\noindent We have the following immediate result from Theorem \ref{Th-2.3}.
\begin{corollary}\label{Cor-2.3}
		For $\alpha=\beta=0$, let $f\in\mathcal{SP}^0_0(0)\subset \mathcal{C}$ be of the form \eqref{Eq-1.1} for all $z\in\mathbb{D}$, then the pre-Schawarzian norm
	\begin{align*}
		||Pf||\leq 2.
	\end{align*}
	The inequality is sharp.
\end{corollary}
\subsection*{Sharpness of Corollary \ref{Cor-2.3}}
For $\alpha=0, \beta=0$, it follows from  that

\begin{align*}
	\frac{f^{\prime\prime}_0(z)}{f_0^{\prime}(z)}=\frac{2z}{1-z^2}\;\;\mbox{and}\;\;Pf_0=\frac{2}{1-z^2}.
\end{align*}
A simple computation thus yields that
\begin{align*}
	||Pf_0||=\sup_{z\in\mathbb{D}}\left(1-z^2\right)|Pf_0|=\sup_{z\in\mathbb{D}}\left(1-|z|^2\right)\frac{2}{1-|z|^2}=2
\end{align*}
and we see the constant $2$ is sharp.\vspace*{2mm}

In our next result, we give a sharp bound for the norm of the Schwarzian derivative when $f\in\mathcal{SP_\alpha}^0(\beta)=\{f\in\mathcal{SP_\alpha}(\beta):f^{\prime\prime}(0)=0\}$ by a direct application of the Schwarz lemma.
\begin{theorem}\label{Th-2.4}
		For $0\leq\beta<1$ and $-\pi/2<\alpha<\pi/2$ let $f\in\mathcal{SP}^{0}_\alpha(\beta)$ be of the form \eqref{Eq-1.1} for all $z\in\mathbb{D}$,  then the Schwarzian norm 
		\begin{align*}
			||Sf||=(1-|z|^2)^2|Sf(z)|\leq2(1-\beta)\cos\alpha\left(2-(1-\beta)\cos\alpha\right).
		\end{align*}
		The inequality is sharp.
	\begin{proof}[\bf Proof of Theorem \ref{Th-2.4}]
		From \eqref{Eq-2.3}, we have
		\begin{align*}
			\frac{f^{\prime\prime}(z)}{f^{\prime}(z)}=\frac{2G_1(\alpha, \beta)\phi(z)}{(1-z\phi(z))}.
		\end{align*}
		A simple calculation shows that
		\begin{align}
			Sf(z)=2G_1(\alpha, \beta)\left(\frac{2\phi^{\prime}(z)+\left(2-2G_1(\alpha, \beta)\right)\phi^2(z)}{2(1-z\phi(z))^2}\right).\nonumber
		\end{align}
		By using triangle inequality and Schwarz pick lemma, we obtain
		\begin{align}\label{Eq-2.22}
			(1-|z|^2)^2|Sf|&\leq 2|G_1(\alpha, \beta)|\bigg|2\phi^{\prime}(z)+\left(2-2G_1(\alpha, \beta)\right)\phi^2(z)\bigg|\frac{(1-|z|^2)^2}{2|1-z\phi(z)|^2}\\&=\frac{2|G_1(\alpha, \beta)|(1-|z|^2)^2}{|1-z\phi(z)|^2}\left(\frac{1-|\phi(z)|^2}{1-|z|^2}+\left(1-|G_1(\alpha, \beta)|\right)|\phi(z)|^2\right).\nonumber
		\end{align}
		We define the function $\Psi(z):\mathbb{D}\rightarrow\mathbb{D}$ such that
		\begin{align*}
			\Psi(z):=\frac{\bar{z}-\phi(z)}{1-z\phi(z)}.
		\end{align*}
		Since $\phi(\mathbb{D})\subseteq\mathbb{D}$ then $(1-|z|^2)(1-|z\phi(z)|^2)>0$, it follows that 
		\begin{align*}
			|\bar{z}-\phi(z)|^2<|1-z\phi(z)|^2.
		\end{align*}
		Hence, we can conclude  that $|\Psi(z)|^2<1$. A simple computation leads to
		\begin{align*}
			1-|\Psi(z)|^2=\frac{(1-|\phi(z)|^2)(1-|z|^2)}{|1-z\phi(z)|^2}
		\end{align*}
		and
		\begin{align}\label{Eq-2.23}
			\frac{(1-|z|^2)^2}{|1-z\phi(z)|^2}=\frac{(1-|\Psi(z)|^2)(1-|z|^2)}{(1-|\phi(z)|^2)}.
		\end{align}
		If we replace the expression$\eqref{Eq-2.23}$ in $\eqref{Eq-2.22}$, we have
		\begin{align}\label{Eq-2.24}
			(1-|z|^2)^2|Sf(z)|\leq 2|G_1(\alpha, \beta)|(1-|\Psi(z)|^2)\left(1+\left(1-|G_1(\alpha, \beta)|\right)\frac{|\phi(z)|^2(1-|z|^2)}{(1-|\phi(z)|^2)}\right).
		\end{align}
		Since $h^{\prime\prime}(0)=0$ implies that $\phi(0)=0$, using Lemma A, we obtain
		\begin{align}\label{Eq-2.25}
			\frac{|\phi(z)|^2}{1-|\phi(z)|^2}\leq\frac{|z|^2}{1-|z|^2}.
		\end{align}
		Using \eqref{Eq-2.25} in \eqref{Eq-2.24}, we obtain
		\begin{align*}
			(1-|z|^2)^2|Sf(z)|\leq 2|G_1(\alpha, \beta)|(1-|\Psi(z)|^2)\left(1+\left(1-|G_1(\alpha, \beta)|\right)|z|^2\right).
		\end{align*}
		Again, since $1-|\Psi(z)|^2$$\leq$1, then
		\begin{align*}
			\sup_{z\in\mathbb{D}}\;	(1-|z|^2)^2|Sf(z)|&\leq\sup_{z\in\mathbb{D}}\;2(1-\beta)\cos\alpha\left(1+\left(1-(1-\beta)\cos\alpha\right)|z|^2\right)\\&=2(1-\beta)\cos\alpha\left(2-(1-\beta)\cos\alpha\right).
		\end{align*}
		Next part of the proof is to show that the inequalities are sharp. The family of parameterized functions defined as:
		\begin{align}
			f_{\alpha,\beta} (z)=\int_{0}^{z}\frac{1}{(1-\xi^2)^{(1-\beta)\cos\alpha}} d\xi,\;\;\;\;\mbox{for}\;\;-\pi/2<\alpha<\pi/2,\;0\leq \beta<1
		\end{align}
		maximizes the Schwarzian norm defined as:
		\begin{align*}
			||Sf||=\sup_{z\in\mathbb{D}}(1-|z|^2)^2|Sf| 
		\end{align*}
		and from this, the sharpness of the inequality holds for $-\pi/2<\alpha<\pi/2,\;\; 0\leq \beta<1$. Note that
		\begin{align}
			\begin{cases}
				\displaystyle \frac{f^{\prime\prime}_{\alpha,\beta}(z)}{f_{\alpha,\beta}^{\prime}(z)}=\frac{2z(1-\beta)\cos\alpha}{1-z^2},\vspace{2mm}\\
				\displaystyle Sf_{\alpha,\beta}=\frac{2(1-\beta)\cos\alpha}{(1-z^2)^2}\left(1+\left(1-(1-\beta)\cos\alpha\right)|z|^2\right)
			\end{cases}
		\end{align}
		which calculates
		\begin{align*}
			||Sf_{\alpha,\beta}||&=\sup_{z\in\mathbb{D}}(1-|z|^2)^2|Sf_{\alpha,\beta}|\\&=\sup_{z\in\mathbb{D}}2(1-\beta)\cos\alpha\left(1+\left(1-(1-\beta)\cos\alpha\right)|z|^2\right)\\&\leq2(1-\beta)\cos\alpha\left(2-(1-\beta)\cos\alpha\right).
		\end{align*}
		\;\;\;\;In general, the integral formula for $f_\alpha$ given in above does not give primitives in terms of elementary functions, however when $\alpha=0$ and also $\beta=0$, we have
		\begin{align*}
			f_{0,0} (z)=\int_{0}^{z}\frac{1}{(1-\xi^2)} d\xi=\frac{1}{2}\log\left(\frac{1+z}{1-z}\right),
		\end{align*}
		where $||Sf_{0,0}||=2$.
	\end{proof}
\end{theorem}

\begin{corollary}
		If $f\in\mathcal{SP}^{0}_\alpha(\beta)$, then for all $z\in\mathbb{D}$ and $\alpha=0,\beta=0$, we have
		\begin{align*}
			||Sf||\leq 2.
		\end{align*}
		The inequality is sharp.
\end{corollary}
Without requiring \( |f''(0)| \) to be zero, we derive a bound for \( (1 - |z|^2)^2 |Sf(z)| \) for functions in $\mathcal{SP_\alpha}(\beta)$.
\begin{theorem}\label{Th-2.5}
	If $f\in\mathcal{SP_\alpha}(\beta)$, for all $z\in\mathbb{D}$ and $-\pi/2<\alpha<\pi/2,\;\;0\leq \beta<1$ and 
	\begin{align}
		\xi=|\phi(0)|=\frac{|f^{\prime\prime}(0)|}{2(1-\beta)\cos\alpha},
	\end{align}
	 then 
	\begin{align*}
		(1-|z|^2)^2|Sf(z)|\leq2(1-\beta)\cos\alpha\left(2+(1-\beta)\cos\alpha\frac{(\xi+|z|)^2}{(1-\xi^2)}\right).
	\end{align*}
\end{theorem}
\begin{corollary}
	If $f\in\mathcal{SP}_0(0):=\mathcal{C}$, for all $z\in\mathbb{D}$ with
	\begin{align}
		\xi=|\phi(0)|=\frac{|f^{\prime\prime}(0)|}{2},
	\end{align}
	then inequality
	\begin{align*}
		(1-|z|^2)^2|Sf(z)|\leq2.
	\end{align*}
	 
\end{corollary}
\begin{proof}[\bf Proof of Theorem \ref{Th-2.5}]
	Let $\xi=|\phi(0)|$. Applying the Lemma A, we  calculate
	\begin{align}\label{Eq-2.30}
			\frac{|\phi(z)|^2}{1-|\phi(z)|^2}\leq\frac{(\xi+|z|)^2}{(1-\xi^2)(1-|z|^2)}.
	\end{align}
	If we substitute $\eqref{Eq-2.30}$ in $\eqref{Eq-2.24}$ , we obtain
	\begin{align}
		(1-|z|^2)^2|Sf(z)|\leq 2(1-\beta)\cos\alpha(1-|\Phi_1(z)|^2)\left(2+(1-\beta)\cos\alpha\frac{(\xi+|z|)^2}{(1-\xi^2)}\right).
	\end{align}
	From the fact that $|z|<1$ and  $1-|\Phi_1(z)|^2\leq1$, we can easily calculate 
	\begin{align}
		(1-|z|^2)^2|Sf(z)|\leq2(1-\beta)\cos\alpha\left(2+(1-\beta)\cos\alpha\frac{(\xi+|z|)^2}{(1-\xi^2)}\right).
	\end{align}
	This completes the proof.
\end{proof}
\section{\bf Radius Problem for genaralized Robertson class $\mathcal{SP_\alpha}(\beta)$}\label{Sec-3}
Determining the radius of convexity and the radius of concavity for a given class of functions, and showing the sharpness of these radii, is an important aspect of Geometric Function Theory.\vspace{1.2mm}

In this section, we intend to answer the following problems.
\begin{problem} \label{Pro-3.1}
	Determine the radius of concavity for the class $\mathcal{SP_\alpha}(\beta)$?
\end{problem}
\begin{problem}\label{Pro-3.2}
	Determine the radius of convexity for the class $\mathcal{SP_\alpha}(\beta)$?
\end{problem}
We investigate the radius of concavity and convexity for a certain class of functions, providing affirmative answers to Problems \ref{Pro-3.1} and \ref{Pro-3.2}.
In this section, we find a lower bound of the radius of concavity $ R_{\rm Co(p)} $ of  the class $ \mathcal{S}(p) $. Now, we consider functions $f$ in $ \mathcal{A} $ that map $ \mathbb{D} $ conformally onto a domain whose complement with respect to $ \mathbb{C} $ is convex and that satisfy the normalization $ f(1)=\infty $. We will denote these families of functions by $ {\rm Co}(A) $. Now $ f\in {\rm Co}(A) $ if, and only if, $ {T_f(z)}>0 $ for every $ z\in\mathbb{D} $, where $ f(0)=f^{\prime}(0)-1 $ and 
\begin{align}\label{Eq-4.1A}
	T_f(z)=\frac{2}{A-1}\left(\frac{(A+1)}{2}\left(\frac{1+z}{1-z}\right)-1-z\frac{f^{\prime\prime}(z)}{f^{\prime}(z)}\right),
\end{align}
where $A\in (1, 2]$.\vspace{2mm}
Inspired by \cite[Definition 1.1.]{Bhowmik-CMFT-2024}, for an arbitrary family $\mathcal{F}$ of functions, we define the radius of concavity.
\begin{definition}
	The radius of concavity (w.r.t $ \mathcal{F} $), a subclass of $ \mathcal{A} $ is the largest number $ \mathrm{R}_{\mathcal{F}}\in (0,1] $  such that for each function $ f\in\mathcal{F} $, $ {\rm Re} \left(T_f(z)\right)>0 $ for all $ |z|< \mathrm{R}_{\mathcal{F}}$, where $ T_{f}(z) $ is defined in \eqref{Eq-1.3}.
\end{definition}
\begin{theorem}
	If $f\in\mathcal{SP_\alpha}(\beta)$, for all $z\in\mathbb{D}$ and $-\pi/2<\alpha<\pi/2,\;\;0\leq \beta<1$ then ${\rm Re}\;\left(T_{f(z)}\right)>0$ for $|z|<\mathrm{R_{\alpha,\beta,{Co(A)}}}$, where $\mathrm{R_{\alpha,\beta,{Co(A)}}}$ is the least value of $r\in(0,1)$ satisfying $\Phi_A(r)=0$ with
	\begin{align*}
		\Phi_A(r)=(A+1-2(1-\beta)\cos\alpha)r^2-2(A+1+(1-\beta)\cos\alpha)r+A-1.
	\end{align*}
	The radius $\mathrm{R_{\alpha,\beta,{Co(A)}}}$ is best possible.
\end{theorem}
\begin{proof}
	In view of Theorem \ref{Th-2.1} relation \emph{(iii)}, we obtain
	\begin{align*}
		\bigg|	(1-|z|^2)\left(\frac{f^{\prime\prime}(z)}{f^{\prime}(z)}\right)-2(1-\beta)\cos\alpha\bar{z}\bigg|\leq (1-\beta)\cos\alpha.
	\end{align*}
	A simple computation shows that
	\begin{align}\label{Eq-4.2}
		1-\frac{(1-\beta)\cos\alpha\;r}{1+r}\leq{\rm Re}\left(1+\frac{zf^{\prime\prime}(z)} {f^{\prime}(z)}\right)\leq1+\frac{(1-\beta)\cos\alpha\;r}{1-r}
	\end{align}
	Then, by the inequality \eqref{Eq-4.1A}, we have
	\begin{align*}
		{\rm Re}\;\left(T_{f(z)}\right)&\geq \frac{2}{A-1}\left(\frac{A+1}{2}\frac{1-r}{1+r}-1-\frac{(1-\beta)\cos\alpha\;r}{1-r}\right)\\&=\frac{(A+1-2(1-\beta)\cos\alpha)r^2-2(A+1+(1-\beta)\cos\alpha)r+A-1}{(A-1)(1-r^2)}\\&=\frac{\Phi_A(r)}{(A-1)(1-r^2)},
	\end{align*}
	where $\Phi_A(r)$ is given in the statement of theorem.\vspace{2mm}
	
	The right hand side of the above inequality is strictly positive if $|z|<\mathrm{R_{\alpha,\beta,{Co(A)}}}$, where  $\mathrm{R_{\alpha,\beta,{Co(A)}}}$ is given in the statement of the theorem. We now investigate the existence of the root $\mathrm{R_{\alpha,\beta,{Co(A)}}}\in (0, 1)$ for each $A\in(1,2]$.\vspace{1.2mm} 
	
	We see that the function $\Phi(r)$ which is defined in the statement of the theorem is continuous on $[0,1]$ with
	\begin{align*}
		\Phi(0)=A-1>0\;\mbox{and}\;\mbox{and}\;\;\;\Phi(1)=-2-4(1-\beta)\cos\alpha<0;\;\;\mbox{for all}\;\alpha,\;\beta.
	\end{align*}
	By the IVT, $\Phi(r)$ has at least one root in $(0,1)$. Hence, 	${\rm Re}\left(T_{f(z)}\right)>0$ if $|z|=r<\mathrm{R_{\alpha,\beta,{Co(A)}}}$ exists for every $A\in(1,2]$. Moreover, if we consider 
	\begin{align*}
		f^{\prime}_{\alpha,\beta}(z)=\frac{1}{(1-z)^{(1-\beta)\cos\alpha}}\; \mbox{for}\; z\in\mathbb{D}\; \mbox{with}\;\beta\in[0,1).
	\end{align*}
	then for this function we compute
	\begin{align*}
		T_{f_\beta}(z)=\frac{2}{A-1}\left(\frac{A+1}{2}\left(\frac{1-z}{1+z}\right)-1-\frac{(1-\beta)\cos\alpha\;z}{(1-z)}\right).
	\end{align*} 
	We observe that, if $z=-r$ and $\mathrm{R_{{\alpha,\beta},{Co(A)}}}<|z|<1$, then ${\rm Re}\;T_{f_{\alpha,\beta}}(z)<0$. This proves the sharpness of the radius $\mathrm{R_{\alpha,\beta,{Co(A)}}}$. This completes the proof.
\end{proof}
\begin{definition}
	The number $r\in [0, 1]$ is called the radius of convexity of a particular subclass $\mathcal{F}_\beta$ of the class $\mathcal{A}$ of normalized analytic functions (where $f(z) = z + \sum_{n=2}^{\infty} a_n z^n$) in the unit disk $\mathbb{D}$ if $r$ is the largest number such that the function $f$ is convex in the disk $|z|<r$.
\end{definition}
A function $f$ analytic in a region $\Omega$ is convex in $\Omega$ if it maps $\Omega$ onto a convex region. For an analytic function $f$ in the unit disk $\mathbb{D}$, the condition for $f$ to be locally univalent and convex in a disk $|z| <r$ is given by the inequality you provided:
\begin{align}\label{Eq-5.3}
	{\rm Re}\left(1+\frac{zf^{\prime\prime}(z)} {f^{\prime}(z)}\right) > 0, \quad \text{for } |z| <r.
\end{align}
\begin{theorem}
	The radius of convexity for the class of function $\mathcal{SP_\alpha}(\beta)$ is at least $\frac{1}{(1-\beta)\cos\alpha-1}$.
\end{theorem}
\begin{proof}
	It follows from the left-hand inequality in \eqref{Eq-4.2} that
	\begin{align*}
		{\rm Re}\left(1+\frac{zf^{\prime\prime}(z)} {f^{\prime}(z)}\right)&\geq1-\frac{(1-\beta)\cos\alpha\;r}{1+r}\\&=\frac{r\left(1-(1-\beta)\cos\alpha\right)+1}{1+r}>0,
	\end{align*}
	when $\frac{1}{(1-\beta)\cos\alpha-1}<r<1$. Thus the radius of convexity for $\mathcal{SP_\alpha}(\beta)$ is at least $\frac{1}{(1-\beta)\cos\alpha-1}$.\vspace{1.2mm}

	To show that this radius sharp, let us consider the function $f^{\prime}_{\alpha,\beta}\in\mathcal{SP_\alpha}(\beta)$, given by
	\begin{align*}
		f^{\prime}_{\alpha,\beta}(z)=\frac{1}{(1+z)^{(1-\beta)\cos\alpha}}\; \mbox{for}\; z\in\mathbb{D}\; \mbox{with}\;\beta\in[0,1).
	\end{align*}
	A simple computation shows that
	\begin{align*}
		{\rm Re}\left(1+\frac{zf_{\alpha,\beta}^{\prime\prime}(z)} {f_{\alpha,\beta}^{\prime}(z)}\right)&=1+\frac{(1-\beta)\cos\alpha\;r}{1-r}\\&=\frac{-r\left(1-(1-\beta)\cos\alpha\right)+1}{1-r}
	\end{align*}
	which shows that the radius is sharp. This completes the proof.
\end{proof}
\noindent{\bf Acknowledgment.} The authors would like to sincerely thank the referee(s) for their helpful suggestions and constructive comments, which significantly improved the presentation of the paper.\vspace{1.2mm}

\noindent\textbf{Compliance of Ethical Standards:}\vspace{1.2mm}

\noindent\textbf{Conflict of interest.} The authors declare that there is no conflict  of interest regarding the publication of this paper.\vspace{1.5mm}

\noindent\textbf{Data availability statement.}  Data sharing is not applicable to this article as no datasets were generated or analyzed during the current study.

    \end{document}